\documentclass[10pt,journal]{article}
\usepackage{cite}
\usepackage{graphicx}
\usepackage{epstopdf}

\usepackage{amsmath, amssymb, amsfonts}

\usepackage{amsthm}
\usepackage{comment}
\usepackage{color}
\newtheorem{problem}{Problem}
\numberwithin{problem}{section}
\newtheorem{theorem}{Theorem}
\numberwithin{theorem}{section}

\numberwithin{assumption}{section}
\newtheorem{lemma}{Lemma}
\numberwithin{lemma}{section}

\numberwithin{corollary}{section}

\DeclareMathOperator*{\argmin}{arg\,min}

\newcommand{\rhv}{{\hat{V}_{t}^{\frac{1}{2}}}}

\usepackage{algorithm, algorithmic}
\usepackage{url}
\usepackage{bm}

\begin{document}
\title{Conjugate-gradient-based Adam for stochastic optimization and its application to deep learning}
\author{Yu~Kobayashi\thanks{Y. Kobayashi and H. Iiduka are with the Department
of Computer Science, Meiji University, Kanagawa 214-8571, Japan.\protect\\
E-mail: yuukbys@cs.meiji.ac.jp}~\and~Hideaki~Iiduka\footnotemark[1]
\date{}}
\maketitle

\begin{abstract}
This paper proposes a conjugate-gradient-based Adam algorithm blending Adam with nonlinear conjugate gradient methods and shows its convergence analysis.
Numerical experiments on text classification and image classification show that the proposed algorithm can train deep neural network models in fewer epochs than the existing adaptive stochastic optimization algorithms can.
\end{abstract}


\section{Introduction}
\label{sec:introduction}
Adaptive stochastic optimization algorithms based on stochastic gradient and exponential moving averages have a strong presence in the machine learning field. 
The algorithms are used to solve stochastic optimization problems; they especially, play a key role in finding more suitable parameters for deep neural network (DNN) models by using empirical risk minimization (ERM) \cite{cesa2004}.
 The DNN models perform very well in many tasks, such as natural language processing (NLP), computer vision, and speech recognition. 
For instance, recurrent neural networks (RNNs) and their variant long short-term memory (LSTM) are useful models that have shown excellent performance in NLP tasks. 
Moreover, convolutional neural networks (CNNs) and their variants such as the residual network (ResNet) \cite{he2016deep} are widely used in the image recognition field \cite{he2018mask}. 
However, these DNN models are complex and need to tune a lot of parameters to optimize, and finding appropriate parameters for the prediction is very hard. 
Therefore, it would be very useful to look for optimization algorithms for minimizing the loss function and finding better parameters.

In this paper, we focus on adaptive stochastic optimization algorithms based on stochastic gradient and exponential moving averages. 
The stochastic gradient descent (SGD) algorithm \cite{borkar2009stochastic,bottou-98x, nemirovski2009robust, robbins1985stochastic}, which uses a stochastic gradient with a smart approximation method, is a great cornerstone that underlies other modern stochastic optimization algorithms.
Numerous variants of SGD have been proposed for many interesting situations, in part, because it is sensitive to an ill-conditioned objective function or step size (called the learning rate in machine learning). 
To deal with this problem, momentum SGD \cite{qian1999momentum} and the Nesterov accelerated gradient method \cite{doki1983nesterov} leverage exponential moving averages of gradients. 
In addition, adaptive methods, AdaGrad \cite{duchi2011adaptive} and RMSProp \cite{tieleman2012lecture}, take advantage of an efficient learning rate derived from element-wise squared stochastic gradients. In the deep learning community, Adam \cite{kingma2014adam} is a popular method that uses exponential moving averages of stochastic gradients and of element-wise squared stochastic gradients. 
However, despite it being a powerful optimization method, Adam does not converge to the minimizers of the stochastic optimization problems in some cases. 
As a result, a variant, AMSGrad \cite{reddi2019convergence}, was proposed to guarantee convergence to the optimal solution.

The nonlinear conjugate gradient (CG) method \cite{hager2006survey} is an elegant, efficient technique of deterministic unconstrained nonlinear optimization. 
Unlike the basic gradient descent methods, the CG method does not use vanilla gradients of the objective function as the search directions. 
Instead of normal gradients, conjugate gradient directions are used in the CG method, which can be computed from not only the current gradient but also past gradients. 
Interestingly, the method requires little memory and has strong local and global convergence properties. 
The way of generating conjugate gradient directions has been researched for long time, and efficient formulae have been proposed, such as Hestenes-Stiefel (HS) \cite{hs}, Fletcher-Reeves (FR) \cite{fletcher1964function}, Polak-Ribi\`ere-Polyak (PRP) \cite{polak1969note, polyak1969conjugate}, Dai-Yuan (DY) \cite{dai1999nonlinear}, and Hager-Zhang (HZ) \cite{hager2005new}.

For the present study, we developed a stochastic optimization algorithm, which we refer to as conjugate-gradient-based Adam (CoBA, Algorithm \ref{algo:coba}), for determining more comfortable parameters for DNN models. 
The algorithm proposed herein combines the CG method with the existing stochastic optimization algorithm, AMSGrad, which is based on Adam. 
Our analysis indicates that the theoretical performance of CoBA is comparable to that of AMSGrad (Theorem \ref{thm:ca}). 
In addition, we {give} several examples {in which} the proposed algorithm can be {used} to train DNN models {in certain} significant tasks. 
In concrete terms, {we conducted numerical experiments on} training {an} LSTM for text classification and making {a} ResNet for image classification. The results demonstrate that{, thanks to the benefits of conjugate gradients,} CoBA performs better than the existing adaptive methods, such as AdaGrad, RMSProp, Adam, and AMSGrad, in the sense of minimizing the sum of loss functions.

This paper is organized as follows. Section~\ref{sec:mp} gives the mathematical preliminaries. Section~\ref{sec:pa} presents the CoBA algorithm for solving the stochastic optimization problem and analyzes its convergence. Section~\ref{sec:exp} numerically compares the behaviors of the proposed algorithms with those of the existing ones. Section~\ref{sec:ccl} concludes the paper with a brief summary.

\section{Mathematical Preliminaries}
\label{sec:mp}
We use the standard notation \(\mathbb{R}^N\) for \(N\)-dimensional Euclidean space, {with} the standard Euclidean inner product \(\langle\cdot ,\cdot\rangle \colon\mathbb{R}^N\times\mathbb{R}^N\to\mathbb{R}\) and associated norm \(\|\cdot \|\). 
{Moreover, for all \(i\in\mathcal{N}:=\{1,\ldots ,N\}\), let \(a_i\) be {the} \(i\)-th coordinate of \(\bm{a}\in\mathbb{R}^N\)}. Then, for all vector{s} \(\bm{a}\in\mathbb{R}^N\), \(\|\cdot\|_\infty\) denotes {the} \(\ell_\infty\)-norm, defined as \(\|\bm{a}\|_\infty :=\max_{i\in\mathcal{N}}|a_i|\), \(\tilde{\bm{a}}\) denotes {the} element-wise square, and \(\text{diag}(\bm{a})\) indicates a diagonal matrix whose diagonal entries starting in the upper left corner are \(a_1,\ldots ,a_N\). Further, for any vectors \(\bm{a}, \bm{b}\in\mathbb{R}^N\), we use \(\max{\{\bm{a}, \bm{b}\}}\) to denote {the} element-wise maximum. For a matrix \(A\) and a constant \(p\), let \(A^{p}\) be {the} element-wise \(p\)-th power of \(A\).

We use \(\mathcal{F}\subset\mathbb{R}^N\) to denote a nonempty, closed convex feasible set and say \(\mathcal{F}\) has a bounded diameter \(D_\infty\) if \(\|\bm{x}-\bm{y}\|_{\infty}\le D_{\infty}\) for all \(\bm{x}\),\(\bm{y}\in\mathcal{F}\). Let \(f\colon\mathcal{F}\to\mathbb{R}\) denote a noisy objective function which is differentiable on \(\mathcal{F}\) and \(f_1,\ldots f_T\) be the realization of the stochastic noisy objective function \(f\) at subsequent timesteps \(t\in\mathcal{T}:=\{1,\ldots ,T\}\). For a {positive-}definite matrix \(A\in\mathbb{R}^{N\times N}\), the Mahalanobis norm is defined as \(\|\cdot\|_A:=\sqrt{\langle \cdot , A\cdot \rangle}\) and the projection onto \(\mathcal{F}\) under the norm \(\|\cdot\|_A\) is defined for all \(\bm{y}\in\mathbb{R}^N\) by 
\begin{align*}
	\{\Pi_{\mathcal{F}, A}(\bm{y})\}:=&\argmin_{\bm{x}\in\mathcal{F}}\|\bm{x}-\bm{y}\|_A
	=\argmin_{\bm{x}\in\mathcal{F}}\sqrt{\langle\bm{x}-\bm{y}, A(\bm{x}-\bm{y})\rangle} .
\end{align*}

Let \(\xi\) be a random number whose probability distribution \(P\) is supported on {a} set \(\Xi\subset\mathbb{R}\). Suppose that it is possible to generate independent, identically distributed (iid) numbers \(\xi_1 ,\xi_2,\ldots\) of realization of \(\xi\). We use \(\bm{\mathsf{g}}_t:=\nabla f_{\xi_t}(\bm{x}_t)\) to denote a stochastic gradient of \(f_{\xi_t}\) at \(\bm{x}_t\in\mathbb{R}^N\). 

\subsection{Adaptive stochastic optimization methods for stochastic optimization}
Let us consider the stochastic optimization problem:
\begin{problem}\label{prob:sop}
Suppose that \(\mathcal{F}\subset\mathbb{R}^N\) is nonempty, closed, and convex and \(f_t\colon\mathcal{F}\to\mathbb{R}\) is convex and differentiable for all \(t\in\mathcal{T}\). Then, 
\begin{align}
	\text{minimize }\sum_{t\in\mathcal{T}} f_t(\bm{x})\text{ subject to }\bm{x}\in\mathcal{F}.
\end{align}
\end{problem}
Stochastic gradient descent (SGD) method \cite{borkar2009stochastic,bottou-98x, nemirovski2009robust, robbins1985stochastic} is a basic method based on using {the} stochastic gradient for {solving} Problem \ref{prob:sop}, and {it outperforms} algorithms based on {a} batch gradient. The method generates the sequence {by using} the following update rule:
\begin{equation}\label{eq:sgd}
\bm{x}_{t+1}:=\Pi_\mathcal{F}(\bm{x}_{t} - \alpha_t\bm{\mathsf{g}}_t),
\end{equation}
where \(\alpha_t >0\) and \(\Pi_\mathcal{F}\colon\mathbb{R}^N\to\mathcal{F}\) is the projection onto the set \(\mathcal{F}\) defined as \(\{\Pi_\mathcal{F}(\bm{y})\}:=\argmin_{\bm{x}\in\mathcal{F}}\|\bm{x}-\bm{y}\|\) (\(\bm{y}\in\mathbb{R}^N\)). A diminishing step size \((\alpha_t:=\alpha / \sqrt{t})_{t\in\mathcal{T}}\) for a positive constant \(\alpha\) is typically used for \((\alpha_t)_{t\in\mathcal{T}}\). Also, adaptive algorithms using {an} exponential moving average, which are variants of SGD, are useful for {solving} Problem \ref{prob:sop}. For instance, AdaGrad \cite{duchi2011adaptive}, RMSProp \cite{tieleman2012lecture}, and Adam \cite{kingma2014adam} {perform very well at} minimizing {the} loss functions used in deep learning {applications}. 

In this paper, we focus on Adam{, which is fastest at minimizing the loss function in deep learning.}
For all \(t\in\mathcal{T}\), the algorithm updates {the} parameter \(\bm{x}_t\) {by using} the following update rule: for \(\alpha\in(0,\infty)\), \(\beta_1\), \(\beta_2\), \(\epsilon\in(0,1)\) and \(\bm{m}_0=\bm{v}_0=\bm{0}\),
\begin{align}\label{eq:adam}
	\begin{split}
		&\bm{m}_{t}:=\beta_1\bm{m}_{t-1}+(1-\beta_1)\bm{\mathsf{g}}_t,\\
		&\bm{v}_{t}:=\beta_2\bm{v}_{t-1}+(1-\beta_2)\tilde{\bm{\mathsf{g}}}_t,\ V_t:=\text{diag}(\bm{v}_{t}),\\
		&\bm{\mathsf{d}}_t:=\left[\frac{{m}_{t,1}}{\sqrt{{v}_{t,1}}+\epsilon},\ldots ,\frac{{m}_{t,N}}{\sqrt{{v}_{t,N}}+\epsilon}\right]^\top,\\
		&\bm{x}_{t+1}:=\Pi_{\mathcal{F},V_t^{\frac{1}{2}}}(\bm{x}_{t}-\alpha \bm{\mathsf{d}}_t).
	\end{split}
\end{align}

Although Adam is {an} excellent {choice} for {solving the} stochastic optimization problem, it does not always converge, as shown in \cite[Theorem 3]{reddi2019convergence}. Reference \cite{reddi2019convergence} presented a good variant algorithm of Adam, called AMSGrad, which converges to a solution {to} Problem \ref{prob:sop}. The AMSGrad algorithm is as follows: for {\((\alpha_t)_{t\in\mathcal{T}}\subset (0,\infty)\)}, \((\beta_{1t})_{t\in\mathcal{T}}\subset (0,1)\), \(\beta_2\), \(\epsilon\in(0,1)\) and \(\bm{m}_0=\bm{v}_0=\bm{0}\),
\begin{align}\label{eq:amsgrad}
	\begin{split}
		&\bm{m}_{t}:=\beta_{1t}\bm{m}_{t-1}+(1-\beta_{1t})\bm{\mathsf{g}}_t,\\
		&\bm{v}_{t}:=\beta_2\bm{v}_{t-1}+(1-\beta_2)\tilde{\bm{\mathsf{g}}}_t,\\
		&\hat{\bm{v}}_{t}:=\max\{\hat{\bm{v}}_{t-1},\bm{v}_{t}\},\ \hat{V}_t:=\text{diag}(\hat{\bm{v}}_{t}),\\
		&\bm{\mathsf{d}}_{t}:=\left[\frac{{m}_{t,1}}{\sqrt{\hat{v}_{t,1}}+\epsilon},\ldots ,\frac{{m}_{t,N}}{\sqrt{\hat{v}_{t,N}}+\epsilon}\right]^\top,\\
		&\bm{x}_{t+1}:=\Pi_{\mathcal{F}, \hat{V}_t^{\frac{1}{2}}}\left(\bm{x}_{t}-\alpha_t \bm{\mathsf{d}}_{t}\right).
	\end{split}
\end{align}
\subsection{Nonlinear conjugate gradient methods}
{N}onlinear conjugate gradient (CG) methods \cite{hager2006survey} are {used} for solving deterministic unconstrained nonlinear optimization problems, as formulated below:
\begin{problem}\label{prob:nuop}
Suppose that \(f\colon\mathbb{R}^N\to\mathbb{R}\) is continuously differentiable. Then,
\begin{align}
	\text{minimize~}f(\bm{x})\text{~subject~to~}\bm{x}\in\mathbb{R}^N.
\end{align}
\end{problem}
{The} nonlinear CG method in \cite{hager2006survey} for {solving} Problem \ref{prob:nuop} generates a sequence $(\bm{x}_t)_{t\in\mathcal{T}}$ with an initial point $\bm{x}_1\in\mathbb{R}^N$ and the following update rule:
\begin{align}\label{eq:cgupdate}
	\bm{x}_{t+1}:=\bm{x}_{t}+\alpha_t \bm{d}_t,
\end{align}
where {\((\alpha_t)_{t\in\mathcal{T}}\subset (0,\infty)\)}.
{The} search direction {\((\bm{d}_t)_{t\in\mathcal{T}}\subset\mathbb{R}^N\)} {used in the update rule (\ref{eq:cgupdate})} is called {the} conjugate gradient direction and {is} defined as the follows:
\begin{align}
	\bm{d}_t:=-\bm{g}_t+\gamma_t \bm{d}_{t-1},
\end{align}
where $\bm{g}_t:=\nabla f(\bm{x}_t)$ and \(\bm{d}_0=\bm{0}\). Here, we use $\gamma_t$ to denote the conjugate gradient update parameter, which can be computed {from} the gradient values $\bm{g}_t$ and $\bm{g}_{t-1}$. 
The parameter $\gamma_t$ has been researched for many years because {its value has a large effect on} the nonlinear objective function \(f\). For instance, the following parameters proposed by Hestenes-Stiefel (HS) \(\gamma_t^{\rm HS}\) \cite{hs}, Fletcher-Reeves (FR) \(\gamma_t^{\rm FR}\) \cite{fletcher1964function}, Polak-Ribi{\`e}re-Polyak (PRP) $\gamma_t^{\rm PRP}$ \cite{polak1969note,polyak1969conjugate}, and Dai-Yuan (DY) $\gamma_t^{\rm DY}$ \cite{dai1999nonlinear} are widely used to solve Problem \ref{prob:nuop}:
\begin{align}
	&\gamma_t^{\text{HS}}:=\frac{\langle \bm{g}_t, \bm{y}_t\rangle } {\langle \bm{d}_{t-1}, \bm{y}_t\rangle },\label{eq:hs}\\
	&\gamma_t^{\text{FR}}:=\frac{\| \bm{g}_t\|^2}{\| \bm{g}_{t-1}\|^2},\label{eq:fr}\\
	&\gamma_t^{\text{PRP}}:=\frac{\langle \bm{g}_t, \bm{y}_t\rangle } {\| \bm{g}_{t-1}\|^2},\label{eq:prp}\\
	&\gamma_t^{\text{DY}}:=\frac{\| \bm{g}_t\|^2} {\langle \bm{d}_{t-1}, \bm{y}_t\rangle },\label{eq:dy}
\end{align}
where $\bm{y}_t:=\bm{g}_t-\bm{g}_{t-1}$.

In addition, Hager-Zhang \(\gamma_t^{\rm HZ}\) \cite{hager2005new} {is an improvement on} \(\gamma_t^{\rm HS}\) defined by (\ref{eq:hs}) {that works well on} Problem \ref{prob:nuop}. 
The parameter $\gamma_t^{\rm HZ}$ is computed {as follows}:
\begin{align}\label{eq:hz}
	\gamma_t^{\text{HZ}}:=\frac{\langle \bm{g}_t, \bm{y}_t\rangle } {\langle \bm{d}_{t-1}, \bm{y}_t\rangle }-\lambda \frac{\| \bm{y}_t\|^2} {\langle \bm{d}_{t-1}, \bm{y}_t\rangle^2 }\langle \bm{g}_t, \bm{d}_{t-1}\rangle ,
\end{align}
where $\bm{y}_t:=\bm{g}_t-\bm{g}_{t-1}$ and $\lambda > 1/4$.
\section{Proposed algorithm}\label{sec:pa}
This section presents {the conjugate}-gradient-based Adam (CoBA) algorithm (Algorithm \ref{algo:coba} {is the listing)}. 
The way in which the parameters satisfying steps 7--11 are computed is based on the update rule of AMSGrad \eqref{eq:amsgrad}. 
The existing algorithm computes an momentum parameter $\bm{m}_t$ and an adaptive learning rate parameter $\bm{v}_t$ by using {the} stochastic gradient $\bm{\mathsf{g}}_t$ computed in step 4 for all $t\in\mathcal{T}$. 
We replace {the} stochastic gradients $\bm{\mathsf{g}}_t$ {used} in AMSGrad with conjugate gradients and compute $\bm{m}_t$ with {the} conjugate gradients $\bm{\mathsf{d}}_t$ computed in steps~5--6 for all $t\in\mathcal{T}$. 
Here, {the} conjugate gradient update parameters \(\gamma_t\) are calculated {using} each of (\ref{eq:hs})--(\ref{eq:hz}) for all \(t\in\mathcal{T}\).
\begin{algorithm}
	\caption{Conjugate-gradient-Based Adam (CoBA)}
	\begin{algorithmic}[1]\label{algo:coba}
    		\REQUIRE{\(\bm{x}_1\in\mathcal{F},(f_t)_{t\in\mathcal{T}}, (\alpha_t)_{t\in\mathcal{T}},\subset(0,\infty), (\beta_{1t})_{t\in\mathcal{T}}\subset (0,1), \beta_2,\epsilon\in (0,1),a \in (1,\infty ), M\in(0, \infty ).\)}
		\STATE{\(t\leftarrow 1\)}
		\STATE{\(\bm{m}_0:=\bm{0}, \bm{v}_0:=\bm{0}, \bm{\mathsf{d}}_0:=\bm{0}\)}
		\LOOP
			\STATE{$\bm{\mathsf{g}}_{t}:=\nabla_{\bm{x}} f_t(\bm{x}_{t})$}
			\STATE{$\gamma_t$: conjugate gradient update parameter}
			\STATE{$\bm{\mathsf{d}}_t:=\bm{\mathsf{g}}_t-\frac{M}{t^{a}}\gamma_t\bm{\mathsf{d}}_{t-1}$}
			\STATE{$\bm{m}_{t}:=\beta_{1t}\bm{m}_{t-1}+(1-\beta_{1t})\bm{\mathsf{d}}_t$}
			\STATE{$\bm{v}_{t}:=\beta_2\bm{v}_{t-1}+(1-\beta_2)\tilde{\bm{\mathsf{g}}}_t$}
			\STATE{$\hat{\bm{v}}_{t}:=\max\{\hat{\bm{v}}_{t-1},\bm{v}_{t}\}$, $\hat{V}_t:=\text{diag}(\hat{\bm{v}}_t)$}
			\STATE{$\hat{\bm{\mathsf{d}}}_{t}:=\left[\frac{{m}_{t,1}}{\sqrt{\hat{v}_{t,1}}+\epsilon},\ldots ,\frac{{m}_{t,N}}{\sqrt{\hat{v}_{t,N}}+\epsilon}\right]^\top$}
			\STATE{$\bm{x}_{t+1}:=\Pi_{\mathcal{F}, \hat{V}_t^{\frac{1}{2}}}\left(\bm{x}_{t}-\alpha_t \hat{\bm{\mathsf{d}}}_{t}\right)$}
		\ENDLOOP
	\end{algorithmic}
\end{algorithm}

Furthermore, we {give a} convergence analysis of the proposed algorithm. The proof is given in Appendix \ref {apx:ca}.
\begin{theorem}\label{thm:ca}
Suppose that \((\bm{x}_t)_{t\in\mathcal{T}}\), \((\bm{v}_t)_{t\in\mathcal{T}}\), and \((\bm{\mathsf{d}}_t)_{t\in\mathcal{T}}\) are the sequences generated by Algorithm \ref{algo:coba} with \(\alpha\in (0,\infty)\), \(\alpha_t:=\alpha / \sqrt{t}\), \(\mu:=\beta_1 / \sqrt{\beta_2} < 1\), \(\beta_{11}:=\beta_1\), and \(\beta_{1t}\le\beta_1\) for all \(t\in\mathcal{T}\). Assume that \((\gamma_t)_{t\in\mathcal{T}}\) is bounded, $\mathcal{F}$ has a bounded diameter $D_{\infty}$, and there exist \(G_{\infty},\bar{G}_\infty\in\mathbb{R}\) such that \(G_{\infty}=\max_{t\in\mathcal{T}}\left({\sup_{\bm{x}\in\mathcal{F}}{\|\nabla f_t(\bm{x})\|_{\infty}}}\right)\) and \(\bar{G}_\infty=\max{\{2G_\infty, \max_{t\in\{1,\ldots ,t_0-1\}}\|\bm{\mathsf{d}}_t\|\}}\) for some \(t_0\). Then, for any solution \(\bm{x}^\star\) of Problem \ref{prob:sop}, the regret $R(T):=\sum_{t=1}^T\{f_t(\bm{x}_t)-f_t(\bm{x}^{\star})\}$ satisfies the following inequality:
\begin{align*}
	R(T)\le&\frac{D_\infty^2\sqrt{T}}{\alpha(1-\beta_1)}\sum_{i=1}^N \sqrt{\hat{v}_{T,i}}
	+\frac{D_\infty^2}{2(1-\beta_1)}\sum_{t=1}^T\frac{\beta_{1t}}{\alpha_t}\sum_{i=1}^N \sqrt{\hat{v}_{t,i}}\\
	&+\frac{\alpha\sqrt{1+\log T}}{(1-\beta_1)^2(1-\mu)\sqrt{1-\beta_2}}\sum_{i=1}^N\sqrt{\sum_{t=1}^{T}\mathsf{d}_{t,i}^2}\\
	&+D_\infty \bar{G}_\infty\sum_{t=1}^T \frac{|\gamma_t|}{t^a}.
\end{align*}
\end{theorem}
Theorem \ref{thm:ca} indicates that Algorithm \ref{algo:coba} has the nice property of convergence {of the average regret \(R(T)/T\)}, whereas Adam does not guarantee convergence in {that} sense, as shown in \cite[Theorem 3]{reddi2019convergence}. 
In addition, we can see that the {properties} of Algorithm \ref{algo:coba} shown in Theorem \ref{thm:ca} {are} theoretically almost the same as {those} of AMSGrad \eqref{eq:amsgrad} (see \cite[Theorem 4]{reddi2019convergence}):
\begin{align*}
	R(T)\le&\frac{D_\infty^2\sqrt{T}}{\alpha(1-\beta_1)}\sum_{i=1}^N \sqrt{\hat{v}_{T,i}}+\frac{D_\infty^2}{2(1-\beta_1)}\sum_{t=1}^T\frac{\beta_{1t}}{\alpha_t}\sum_{i=1}^N \sqrt{\hat{v}_{t,i}}\\
	&+\frac{\alpha\sqrt{1+\log T}}{(1-\beta_1)^2(1-\mu)\sqrt{1-\beta_2}}\sum_{i=1}^N\sqrt{\sum_{t=1}^{T}\mathsf{g}_{t,i}^2}.\\
\end{align*}

\section{Experiments}\label{sec:exp}
This section presents the results of experiments evaluating our algorithms and comparing them with the existing algorithms.

Our experiments were conducted on {a fast scalar computation server}\footnote{\url{https://www.meiji.ac.jp/isys/hpc/ia.html}} at Meiji University. 
The environment has two Intel(R) Xeon(R) Gold 6148 (2.4 GHz, 20 cores) CPU{s}, a{n} NVIDIA Tesla V100 (16GB, 900Gbps) GPU and {a} Red Hat Enterprise Linux 7.6 operating system. 
The experimental code {was} written in Python 3.6.9{,} and we used {the} NumPy 1.17.3 package and PyTorch 1.3.0 package.

\subsection{Text classification}
\label{subsec:text-classification}
We used the proposed algorithms to learn a {long short-term memory} (LSTM) for text classification. 
The LSTM is an artificial recurrent neural network (RNN) architecture used in the field of deep learning {for} natural language processing{, time-}series analysis{, etc}. 

This experiment used the IMDb dataset\footnote{\url{https://datasets.imdbws.com/}} for text classification task{s}. 
The dataset contains 50,000 movie reviews along with their associated binary sentiment polarity labels. 
The dataset is split into 25,000 train{ing} and 25,000 test sets.

We trained a multilayer neural network for {solving} the text classification problem on the IMDb dataset. 
We used a{n} LSTM with an {affine} layer and a {sigmoid} function as an activation function for the output. 
For training it, we used the {binary cross entropy} (BCE) as a loss function minimized by {the} existing and proposed algorithms. 
The BCE loss {\(L\colon\mathbb{R}^T\times\mathbb{R}^T\to\mathbb{R}\)} is defined as follows:
\begin{equation}\label{eq:loss}
	L(\bm{y}, \bm{z}):=-\frac{1}{T}\sum_{t=1}^T \Bigl\{y_t\log z_t + (1-y_t)\log (1-z_t)\Bigr\},
\end{equation}
where {\(\bm{y}:=(y_t)_{t\in\mathcal{T}}\)} with a binary class label \(y_t\in\{0,1\}\){, meaning} a positive or negative review, and {\(\bm{z}:=(z_t)_{t\in\mathcal{T}}\)} with the output of the neural network \(z_t\in [0,1]\) at each time step \(t\in\mathcal{T}\).

Let us numerically compare the performances of the proposed algorithms with Adam, AMSGrad, RMSProp, and AdaGrad. 
In this experiment, we used a random vector as the initial parameter $\bm{x}_1$ and \(\alpha_t=\alpha:=10^{-2}\){, for all $t\in\mathcal{T}$,} as the step size parameter of all algorithms. 
The previously reported results (see \cite{iiduka2019ensemble, iiduka2016convergence}) on convex optimization algorithms {empirically} used \(\alpha_t:=10^{-2}\) and \(\alpha_t:=10^{-3}\). 
We {used} the default values provided in {\tt torch.optim}\footnote{\url{https://pytorch.org/docs/stable/optim.html}} as the hyper parameter settings of the optimization algorithms and set $\beta_1:=0.9$ and $\beta_2:=0.999$ in Adam, AMSGrad, and CoBA. We set {\(\lambda:=2\)}, \(M:=10^{-4}\), and \(a:=1+10^{-5}\) in CoBA.

The result{s} of the experiment {are} reported in Figures \ref{fig1}--\ref{fig4}. 
Figure \ref{fig1} shows the behaviors of the algorithms for the loss function values defined by (\ref{eq:loss}) with respect to the number of epochs, while Figure \ref{fig2} shows those with respect to elapsed time [s]. 
Figure \ref{fig3} presents the accuracy scores of the classification on the training data with respect to the number of epochs, whereas Figure \ref{fig4} plots the accuracy score versus elapsed time [s]. 
We can see that the CoBA algorithms perform better than Adam, AdaGrad, and RMSProp in terms of both the train{ing} loss and accuracy score. 
In particular, Figures \ref{fig1} and \ref{fig2} show that CoBA using \(\gamma^{\text{HZ}}\) reduces the loss function values in fewer epochs and shorter elapsed time than AMSGrad. 
Figure \ref{fig3} and \ref{fig4} indicate that CoBA using \(\gamma^{\text{HZ}}\) reaches \(100\%\) {accuracy faster} than AMSGrad.

\begin{figure}[htbp]
\centering
\includegraphics[width=.85\textwidth]{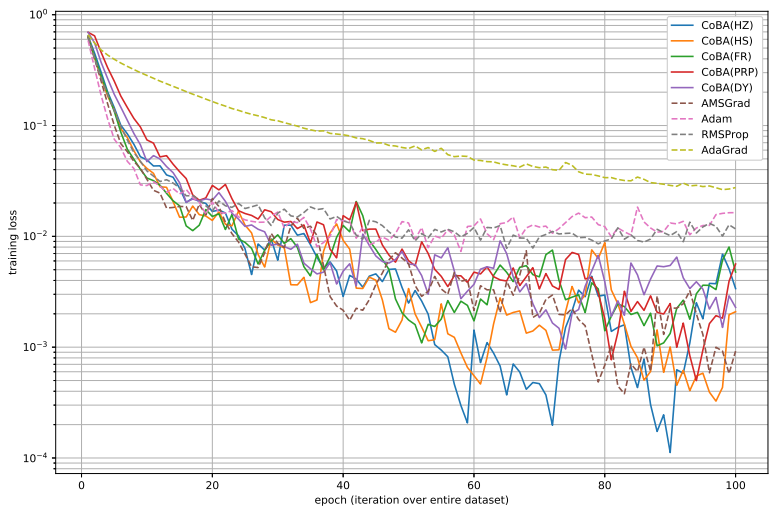}
\caption{\label{fig1}Loss function value versus number of epochs on the IMDb dataset for training.}
\end{figure}

\begin{figure}[htbp]
\centering
\includegraphics[width=.85\textwidth]{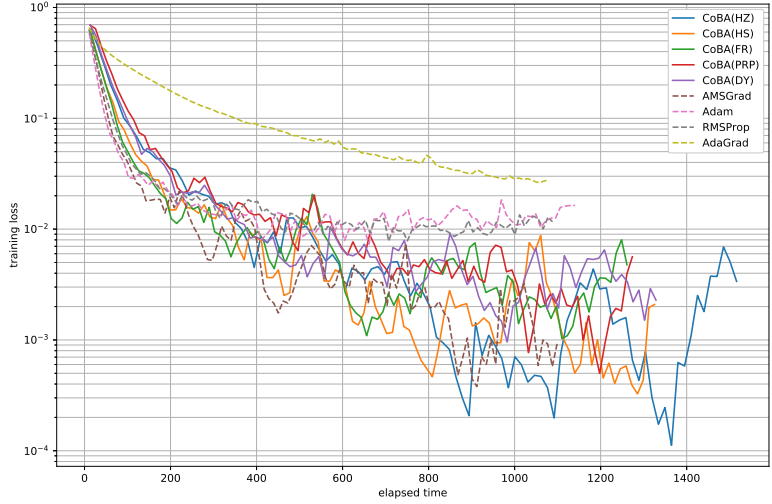}
\caption{\label{fig2}Loss function value {versus} elapsed time [s] on the IMDb dataset for training.}
\end{figure}

\begin{figure}[htbp]
\centering
\includegraphics[width=.85\textwidth]{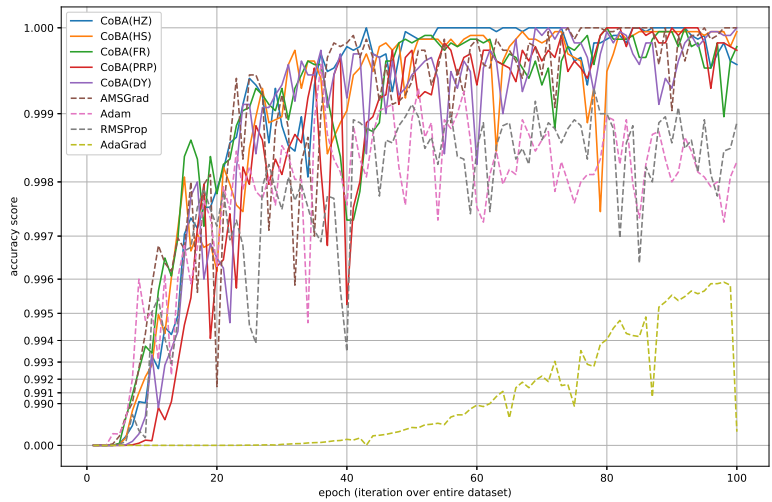}
\caption{\label{fig3}Classification accuracy score versus number of epochs on the IMDb dataset for training.}
\end{figure}

\begin{figure}[htbp]
\centering
\includegraphics[width=.85\textwidth]{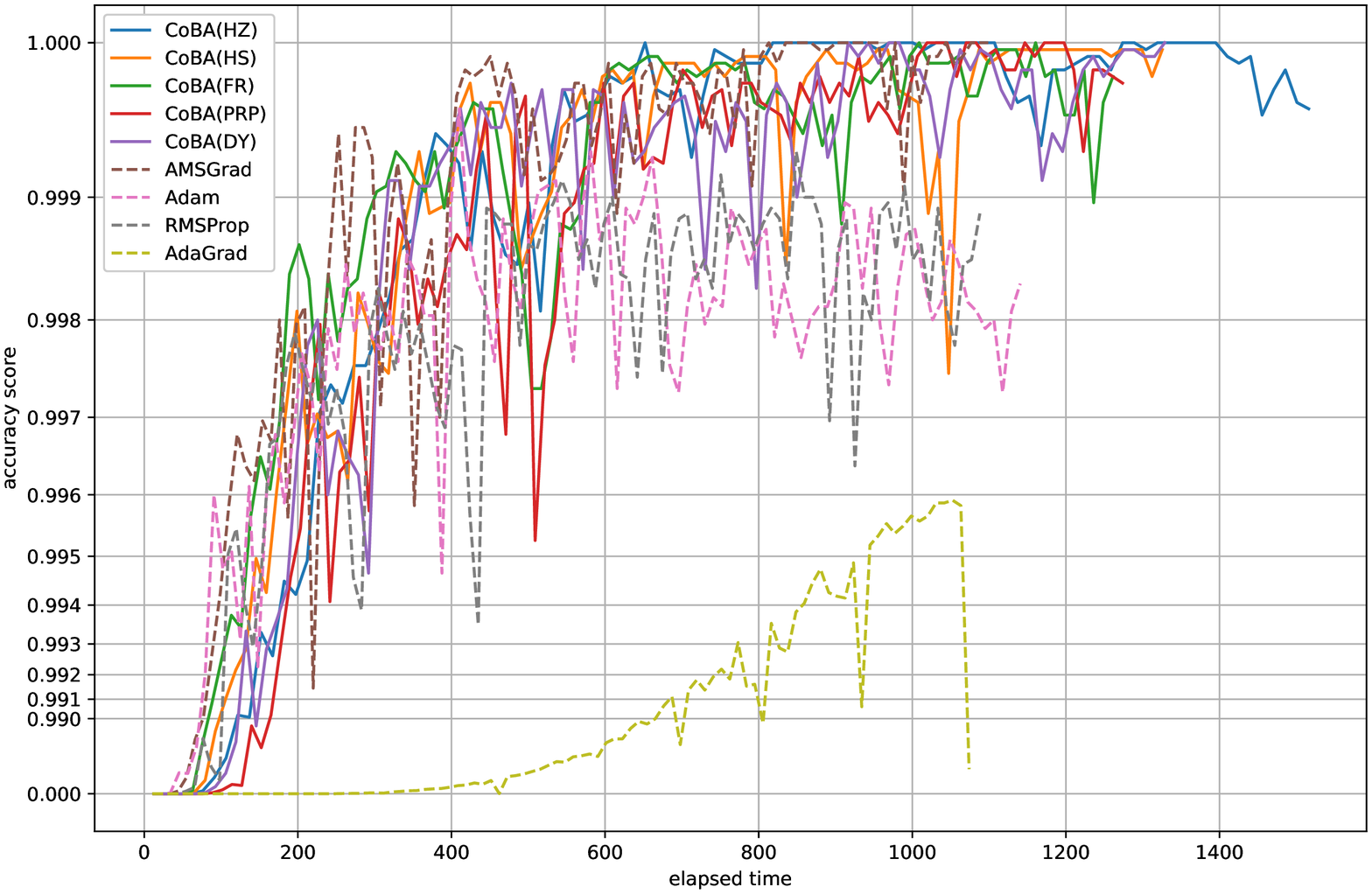}
\caption{\label{fig4}Classification accuracy score versus elapsed time [s] on the IMDb dataset for training.}
\end{figure}

\subsection{Image classification}
{We performed} numerical comparisons using Residual Network (ResNet) \cite{he2016deep}, {a} relatively deep model based on {a convolutional neural network} (CNN{), on an image classification task. Rather than having only convolutional layers,} ResNet {has additional} shortcut connections, e.g., identity mappings, between pairs of 3\(\times\)3 filters. 
The architecture can relieve the {\it degradation} problem {wherein} accuracy saturate{s} when a deeper neural network starts converging. 
As a result, ResNet {is considered to be a practical architecture for} image recognition {on} some datasets. 
In this experiment, we use{d the} CIFAR10 dataset \cite{cifar10}{, a} benchmark for image classification. The dataset consists of 60,000 color images (32\(\times\)32) in 10 classes, with 6,000 images per class. There are 50,000 training images and 10,000 test images. The test batch {contained} exactly 1,000 randomly selected images from each class.

{We trained a} 34-layer ResNet (ResNet-34) organized {into} a 7\(\times\)7 convolutional layer, {32 convolutional layers which have \(3\times3\) filters}, and a 1,000-way-fully-connected layer with {a} softmax function. 
We use{d the cross entropy} as {the} loss function for fitting ResNet {in accordance with the} common strategy in image classification. 
In the case of classification to {the} \(K\)-class, the {cross entropy} {\(L:\mathbb{R}^{T\times K}\times\mathbb{R}^{T\times K}\to\mathbb{R}\)} {\tt torch.nn.CrossEntropyLoss}\footnote{\url{https://pytorch.org/docs/stable/nn.html}} {is} defined as {follows}:
\begin{align}
	L(Y,Z):=-\frac{1}{T}\sum_{t=1}^T\sum_{k=1}^K y_{t,k}\log{z_{t,k}}, \label{eq:ce}
\end{align}
where {\(Y:=(y_{t,k})_{t\in\mathcal{T},k\in\mathcal{K}}\)} with the one-hot multi-class label \(y_{t,k}\in\{0,1\}\) and {\(Z:=(z_{t,k})_{t\in\mathcal{T},k\in\mathcal{K}}\)} with the output of the neural network \(z_{t,k}\in [0,1]\) for all \(t\in\mathcal{T}\) and \(k\in\mathcal{K}:=\{1,\ldots ,K\}\).

In this experiment, we used a random vector as the initial parameter $\bm{x}_1$ and \(\alpha_t=\alpha:=10^{-3}\){, for all $t\in\mathcal{T}$,} as the step size parameter \cite{iiduka2019ensemble, iiduka2016convergence} of all the algorithms. 
As {described in} Subsection \ref{subsec:text-classification}, we set the default values {of Adam, AMSGrad, and CoBA to} $\beta_1:=0.9$ and $\beta_2:=0.999$. 
For each type of conjugate gradient update parameter \(\gamma_t\), we set the coefficients \(M\) and \(a\) to values optimized by a grid search over a parameter grid consisting of \(M\in\{10^{-2},10^{-3},10^{-4}\}\) and \(a\in\{1+10^{-4},1+10^{-5},1+10^{-6},1+10^{-7}\}\). We set {\(\lambda:=2\)} in CoBA(HZ).

The results of the experiments are reported in Figure \ref{fig5}--\ref{fig8}. 
Figure \ref{fig5} {plots the} loss function values defined by \eqref{eq:ce} versus the number epochs, while Figure \ref{fig6} plots the loss function values versus elapsed time [s]. 
Figure \ref{fig7} presents the accuracy score on the dataset for training every epoch, whereas Figure \ref{fig8} plots the accuracy score versus elapsed time [s].

We can see that the CoBA algorithms perform better than Adam, AdaGrad, and RMSProp {in terms of} both the train loss and accuracy score. 
In particular, Figures \ref{fig5} and \ref{fig7} show that CoBA using \(\gamma^{\text{HS}}\), \(\gamma^{\text{FR}}\), \(\gamma^{\text{PRP}}\), \(\gamma^{\text{DY}}\), or \(\gamma^{\text{HZ}}\) reduce{s} the loss function values and {reaches an} accuracy score of \(100\%\) in fewer epochs than AMSGrad. 
Figure \ref{fig6} shows that CoBA and AMSGrad converge faster than {the} other algorithms. 
Although {CoBA} it takes more time than AMSGrad does to update the parameters of ResNet, they theoretically take about the same amount of time for computing the conjugate gradient direction \cite{golub2013matrix}. 
Figures \ref{fig7} and \ref{fig8} indicate that CoBA using \(\gamma^{\text{PRP}}\) reaches \(100\%\) accuracy faster than AMSGrad.

\begin{figure}[htbp]
\centering
\includegraphics[width=.85\textwidth]{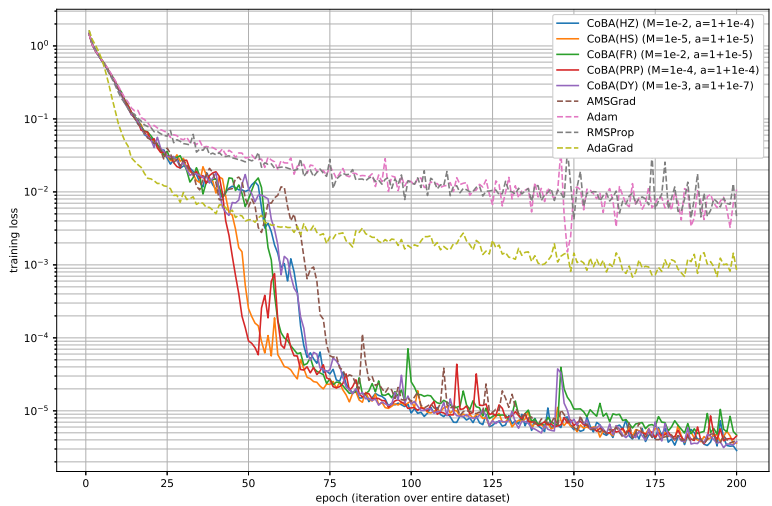}
\caption{\label{fig5}Loss function value {versus number of} epochs on the CIFAR-10 dataset for training.}
\end{figure}

\begin{figure}[htbp]
\centering
\includegraphics[width=.85\textwidth]{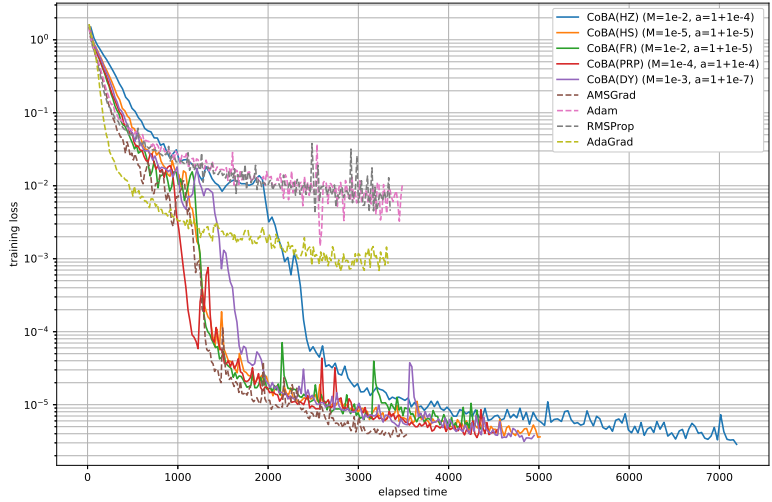}
\caption{\label{fig6}Loss function value versus elapsed time [s] on the CIFAR-10 dataset for training.}
\end{figure}

\begin{figure}[htbp]
\centering
\includegraphics[width=.85\textwidth]{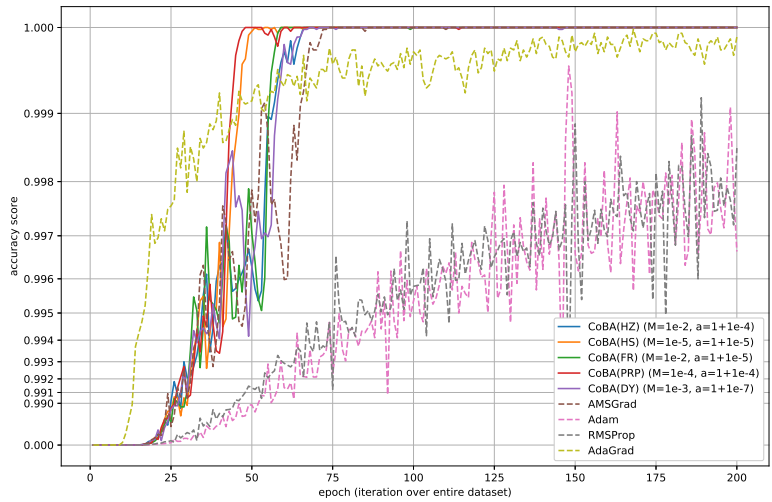}
\caption{\label{fig7}Classification accuracy score versus number of epochs on the CIFAR-10 dataset for training.}
\end{figure}

\begin{figure}[htbp]
\centering
\includegraphics[width=.85\textwidth]{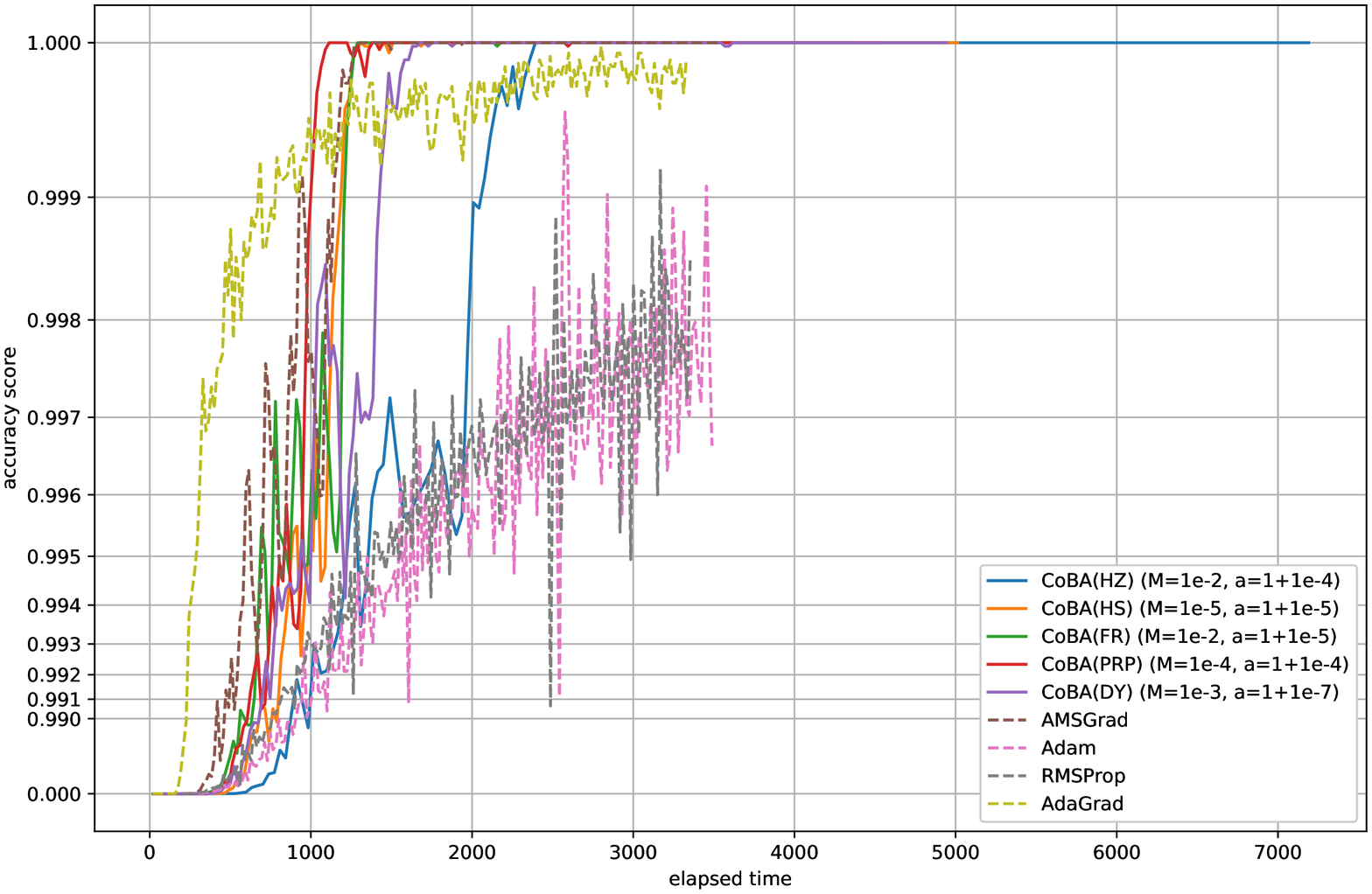}
\caption{\label{fig8}Classification accuracy score versus elapsed time [s] on the CIFAR-10 dataset for training.}
\end{figure}

\section{Conclusion and Future Work}\label{sec:ccl}
We presented the {c}onjugate-gradient-{b}ased Adam (CoBA) algorithm for solving stochastic optimization problem{s that} minimize the empirical risk in fitting {of} deep neural networks and showed its convergence. 
{We} numerically {compared CoBA} with {an} existing learning method {in a} text classification {task using the} IMDb dataset and {an} image classification {task using the} CIFAR-10 dataset. 
The results demonstrated {its} optimality and efficiency. 
In particular, {compared with the existing methods, CoBA reduced} the loss function value in fewer epochs {on both datasets}. 
In addition, it classification score reached a \(100\%\) accuracy in fewer epochs compared with the existing methods.

In the future, we would like to improve the implementation of {the} proposed algorithms to enable {computation of} conjugate gradients in {a} theoretically reasonable time. 
In addition, we would like to design a more appropriate a stochastic conjugate gradient direction and conjugate gradient update parameter, e.g., {one in which the expected value is equivalent to a deterministic conjugate gradient}. Furthermore, we would like to {find a} way to find a suitable step size which permits {the} proposed algorithm to converge {faster} to the {solution to the} stochastic optimization problem.

\appendix
\section{Proof of Theorem\ref{thm:ca}}\label{apx:ca}
To begin with, we {prove} the following lemma with the boundedness of the conjugate gradient direction \(\bm{\mathsf{d}}_t\ (t\in\mathcal{T})\).
\begin{lemma}\label{lem:cg-bdd}
{Suppose} that \((\bm{x}_t)_{t\in\mathcal{T}}\) is the sequence generated by Algorithm \ref{algo:coba} with the parameter settings and conditions assumed in Theorem \ref{thm:ca}. 
{Further, assume} that \((\gamma_t)_{t\in\mathcal{T}}\) is bounded and there exist \(G_{\infty},\bar{G}_\infty\in\mathbb{R}\) such that \(G_{\infty}=\max_{t\in\mathcal{T}}\left({\sup_{\bm{x}\in\mathcal{F}}{\|\nabla f_t(\bm{x})\|_{\infty}}}\right)\) and \(\bar{G}_\infty=\max{\{2G_\infty, \max_{t\in\{1,\ldots ,t_0-1\}}{\|\bm{\mathsf{d}}_t\|\}}}\) for some \(t_0\). Then, \(\|\bm{\mathsf{d}}_t\|\le\bar{G}_\infty\) holds for all \(t\).
\end{lemma}
\begin{proof}
We {will} use mathematical induction. {The fact that \(\frac{M}{t^a}|\gamma_t|\to0\) \((t\to\infty)\) ensures that there exists $t_0 \in \mathbb{N}$ such that, for all \(t \ge t_0\), 
	\begin{align}\label{eq:1}
	\frac{M}{t^a}|\gamma_t|\le\frac{1}{2}.
	\end{align} }
The definition of \(\bar{G}_\infty\) implies that {$\|\bm{\mathsf{d}}_{t} \| \le \bar{G}_\infty$} for all \(t<t_0\). {Suppose} that \(\|\bm{\mathsf{d}}_{j-1}\|\le \bar{G}_{\infty}\) for some \( j\ge t_0 \). Then, from the definition of \(\bm{\mathsf{d}}_{j}\) and the triangle inequality,
	\begin{align*}
		\|\bm{\mathsf{d}}_j\| &\le \|\bm{\mathsf{g}}_j\|+\frac{M}{j^a}|\gamma_j|\|\bm{\mathsf{d}}_{j-1}\|\\
		&\le G_\infty + \frac{M}{j^a}|\gamma_j|\bar{G}_\infty,
	\end{align*}
{which, together with \eqref{eq:1}, implies that} 
	\begin{align*}
		\|\bm{\mathsf{d}}_j\| \le \frac{\bar{G}_\infty}{2} + \frac{\bar{G}_\infty}{2} = \bar{G}_\infty.
	\end{align*}
Accordingly, \(\|\bm{\mathsf{d}}_t\|\le\bar{G}_\infty\) holds for all \(t\).
\end{proof}

Next, we show the following lemma:
\begin{lemma}\label{lem:d-bdd}
For the parameter settings and conditions assumed in Theorem \ref{thm:ca} and for any \(l\in\{0,1,\ldots ,T-2\}\), we have
\begin{align*}
	\begin{split}
		&\sum_{t=l+1}^T \alpha_t \sum_{i=1}^N\frac{m_{t-l, i}^2}{\sqrt{\hat{v}_{t, i}}}\\
		\le &\frac{\alpha \sqrt{1+\log T}}{(1-\beta_1)(1-\mu)\sqrt{1-\beta_2}}
			\sum_{i=1}^N \sqrt{\sum_{t=1}^{T}\mathsf{d}_{t,i}^2}.
	\end{split}
\end{align*}
\end{lemma}

\begin{proof}
Let \(l\in\{0,1,\ldots ,T-2\}\) {be} fixed arbitrarily. For all \( j\in\mathcal{T} \), we define \( \bar{\beta}_{1j} := \prod_{k=j+1}^{T-l}\beta_{1k} \). Then, we have
\begin{align*}
	&\sum_{t=l+1}^T \alpha_t \sum_{i=1}^N\frac{m_{t-l, i}^2}{\sqrt{\hat{v}_{t, i}}}\\
	=&\sum_{t=l+1}^{T-1} \alpha_t \sum_{i=1}^N\frac{m_{t-l, i}^2}{\sqrt{\hat{v}_{t, i}}} + \alpha_T\sum_{i=1}^N\frac{m_{T-l, i}^2}{\sqrt{\hat{v}_{T, i}}}\\
	\le&\sum_{t=l+1}^{T-1} \alpha_t \sum_{i=1}^N\frac{m_{t-l, i}^2}{\sqrt{\hat{v}_{t, i}}} + \alpha_T\sum_{i=1}^N\frac{m_{T-l, i}^2}{\sqrt{v_{T, i}}}\\
	\le&\sum_{t=l+1}^{T-1} \alpha_t \sum_{i=1}^N\frac{m_{t-l, i}^2}{\sqrt{\hat{v}_{t, i}}} + \frac{\alpha}{\sqrt{T}}\sum_{i=1}^N\frac {\left\{\sum_{j=1}^{T-l}(1-\beta_{1j})\bar{\beta}_{1j}\mathsf{d}_{j,i}\right\}^2}{\sqrt{(1-\beta_2) \sum_{j=1}^T\beta_2^{T-j}\mathsf{d}_{j,i}^2 }},\\
	\le&\sum_{t=l+1}^{T-1} \alpha_t \sum_{i=1}^N\frac{m_{t-l, i}^2}{\sqrt{\hat{v}_{t, i}}} + \frac{\alpha}{\sqrt{T(1-\beta_2)}}\sum_{i=1}^N\frac{\left(\sum_{j=1}^{T-l}\bar{\beta}_{1j}\mathsf{d}_{j,i}\right)^2}{\sqrt{\sum_{j=1}^T\beta_2^{T-j}\mathsf{d}_{j,i}^2 }},
\end{align*}
where the second inequality comes from the definition of \(\hat{v}_{T,i}\), which is \(\hat{v}_{T,i}=\max \{ \hat{v}_{T-1, i}, v_{T, i}\} \), the third one follows from \(\alpha_T:=\alpha /\sqrt{T}\) and the update rules of \(m_{T,i}\) and \(v_{T,i}\) in Algorithm \ref{algo:coba} for \(i\in\mathcal{N}\), and the {fourth} one comes from \(\beta_{1j}\le\beta_1\) for all \(j\in\mathcal{T}\). Here, from the Cauchy-Schwarz inequality and the fact that \(\bar{\beta}_{1j}\le\beta_1^{T-j}\le 1\) for all \(j\in\mathcal{T}\), {we have}
\begin{align*}
	&\sum_{i=1}^N\frac {\left(\sum_{j=1}^{T-l}\bar{\beta}_{1j}\mathsf{d}_{j,i}\right)^2}{\sqrt{\sum_{j=1}^T\beta_2^{T-j}\mathsf{d}_{j,i}^2}}\\
	\le &\sum_{i=1}^N\frac{\left(\sum_{j=1}^{T-l}\bar{\beta}_{1j}\right)\left(\sum_{j=1}^{T-l}\bar{\beta}_{1j}\mathsf{d}_{j,i}^2\right)}{\sqrt{\sum_{j=1}^T\beta_2^{T-j}\mathsf{d}_{j,i}^2}}\\
	\le&\frac{1}{1-\beta_1}\sum_{i=1}^N\frac{\sum_{j=1}^T\beta_1^{T-j}\mathsf{d}_{j,i}^2}{\sqrt{\beta_2^{T-j}\mathsf{d}_{j,i}^2}}\\
	\le&\frac{1}{1-\beta_1}\sum_{i=1}^N\sum_{j=1}^T\mu^{T-j}|\mathsf{d}_{j,i}|,
\end{align*}
where \(\mu\) is defined by \(\mu:=\beta_1/\sqrt{\beta_2}\). Hence, we have
\begin{align*}
	&\sum_{t=l+1}^{T}\alpha_t \sum_{i=1}^N\frac{m_{t-l, i}^2}{\sqrt{\hat{v}_{t, i}}}
	\le\sum_{t=l+1}^{T-1}\alpha_t \sum_{i=1}^N\frac{m_{t-l, i}^2}{\sqrt{\hat{v}_{t, i}}}
	+ \frac{\alpha}{(1-\beta_1)\sqrt{1-\beta_2}}\sum_{i=1}^N\sum_{j=1}^T\frac{\mu^{T-j}|\mathsf{d}_{j,i}|}{\sqrt{T}}.
\end{align*}
A discussion similar to the one for all \(t\in\{l+1,\ldots ,T-1\}\) ensures that
\begin{align*}
	&\sum_{t=l+1}^{T}\alpha_t \sum_{i=1}^N\frac{m_{t-l, i}^2}{\sqrt{\hat{v}_{t, i}}}\\
	\le&\sum_{t=l+1}^{T}\frac{\alpha}{(1-\beta_1)\sqrt{1-\beta_2}}\sum_{i=1}^N\sum_{j=1}^t\frac{\mu^{t-j}|\mathsf{d}_{j,i}|}{\sqrt{t}}\\
	=&\frac{\alpha}{(1-\beta_1)\sqrt{1-\beta_2}}\sum_{i=1}^N\sum_{t=l+1}^{T}\sum_{j=1}^t\frac{\mu^{t-j}|\mathsf{d}_{j,i}|}{\sqrt{t}}\\
	=&\frac{\alpha}{(1-\beta_1)\sqrt{1-\beta_2}}\sum_{i=1}^N\sum_{t=l+1}^{T}|\mathsf{d}_{t,i}|\sum_{j=t}^T\frac{\mu^{j-t}}{\sqrt{j}}\\
	\le&\frac{\alpha}{(1-\beta_1)\sqrt{1-\beta_2}}\sum_{i=1}^N\sum_{t=l+1}^{T}|\mathsf{d}_{t,i}|\sum_{j=t}^T\frac{\mu^{j-t}}{\sqrt{t}}\\
	\le&\frac{\alpha}{(1-\beta_1)(1-\mu)\sqrt{1-\beta_2}}\sum_{i=1}^N\sum_{t=l+1}^{T}\frac{|\mathsf{d}_{t,i}|}{\sqrt{t}}.
\end{align*}
From the Cauchy-Schwarz inequality,
\begin{align*}
	&\sum_{t=l+1}^{T}\alpha_t \sum_{i=1}^N\frac{m_{t-l, i}^2}{\sqrt{\hat{v}_{t, i}}}\\
	\le&\frac{\alpha}{(1-\beta_1)(1-\mu)\sqrt{1-\beta_2}}\sum_{i=1}^N\sqrt{\sum_{t=l+1}^{T}\mathsf{d}_{t,i}^2}\sqrt{\sum_{t=l+1}^{T}\frac{1}{t}}\\
	\le&\frac{\alpha\sqrt{1+\log T}}{(1-\beta_1)(1-\mu)\sqrt{1-\beta_2}}\sum_{i=1}^N\sqrt{\sum_{t=1}^{T}\mathsf{d}_{t,i}^2}.
\end{align*}
This completes the proof.
\end{proof}

Finally, we {prove} Theorem \ref{thm:ca}.
\begin{proof}
Let \(\bm{x}^\star\in\argmin_{\bm{x}\in\mathcal{F}}f(\bm{x})\) and \(t\in\mathcal{T}\) be fixed arbitrarily. From the update rule of Algorithm \ref{algo:coba}, {we have}
\begin{align*}
	\{\bm{x}_{t+1}\}&=\left\{\Pi_{\mathcal{F}, \rhv}\left(\bm{x}_{t}-\alpha_t\hat{\bm{\mathsf{d}}}_t\right)\right\}\\
	&=\argmin_{\bm{x}\in\mathcal{F}}\left\| \bm{x}_{t}-\alpha_t\hat{\bm{\mathsf{d}}}_t-\bm{x}^\star\right\|_\rhv.
\end{align*}
Here, for all {positive-}definite matrix{es} \(Q\in\mathbb{R}^{N\times N}\) and for all \(\bm{z}_1,\bm{z}_2\in\mathbb{R}^N\) with \(\bm{u}_1:=\Pi_{\mathcal{F}, Q}(\bm{z}_1), \bm{u}_2:=\Pi_{\mathcal{F}, Q}(\bm{z}_2)\), we have \(\left\|\bm{u}_2-\bm{u}_1\right\|_Q\le\left\|\bm{z}_1-\bm{z}_2\right\|_Q\)\cite[Lemma 4]{reddi2019convergence}. Hence, 
\begin{align*}
	&\left\|\bm{x}_{t+1}-\bm{x}^\star\right\|_\rhv^2
	\le \left\|\bm{x}_{t}-\alpha_t\hat{\bm{\mathsf{d}}}_t-\bm{x}^\star\right\|_\rhv^2,
\end{align*}
which, together with $\|\bm{x}-\bm{y}\|^2 = \|\bm{x}\|^2 + \|\bm{y}\|^2 + 2 \langle \bm{x}, \bm{y} \rangle$ ($\bm{x},\bm{y} \in \mathbb{R}^N$) and the definitions of $\bm{m}_t$ and $\bm{\mathsf{d}}_t$, implies that
\begin{align*} 
	&\left\|\bm{x}_{t+1}-\bm{x}^\star\right\|_\rhv^2\\
	\le&\left\| \bm{x}_{t}-\bm{x}^\star\right\|_\rhv^2 +\alpha_t^2\left\|\hat{\bm{\mathsf{d}}}_t \right\|_\rhv^2\\
	&-2\alpha_t\langle\bm{m}_t ,\bm{x}_{t}-\bm{x}^\star\rangle\\
	\le&\left\| \bm{x}_{t}-\bm{x}^\star\right\|_\rhv^2 +\alpha_t^2\left\|\hat{\bm{\mathsf{d}}}_t \right\|_\rhv^2\\
	&-2\alpha_t\langle\beta_{1t}\bm{m}_{t-1}+(1-\beta_{1t})\bm{\mathsf{d}}_t ,\bm{x}_{t}-\bm{x}^\star\rangle\\
	\le&\left\| \bm{x}_{t}-\bm{x}^\star\right\|_\rhv^2 +\alpha_t^2\left\|\hat{\bm{\mathsf{d}}}_t \right\|_\rhv^2
	-2\alpha_t\beta_{1t}\langle\bm{m}_{t-1},\bm{x}_{t}-\bm{x}^\star\rangle\\
	&-2\alpha_t(1-\beta_{1t})\left\langle\bm{\mathsf{g}}_t -\frac{\gamma_t}{t^a}\bm{\mathsf{d}}_{t-1},\bm{x}_{t}-\bm{x}^\star\right\rangle.
\end{align*}
From the Cauchy-Schwarz and Young {inequalities} with 
\begin{align*}
	\bar{\bm{\mathsf{d}}}_t:=\left[\ {m_{t-1,1}}/{\sqrt{\hat{v}_{t,1}}}, \ldots , {m_{t-1,N}}/{\sqrt{\hat{v}_{t,N}}}\ \right]^\top,
\end{align*}
we get
\begin{align*}
	&-\langle\bm{m}_{t-1} ,\bm{x}_{t}-\bm{x}^\star\rangle\\
	\le&\sqrt{\alpha_t} \left\| \bar{\bm{\mathsf{d}}}_t \right\|_\rhv \cdot \frac{1}{\sqrt{\alpha_t}} \left\| \bm{x}_{t}-\bm{x}^\star \right\|_\rhv\\
	\le&\frac{\alpha_t}{2} \left\| \bar{\bm{\mathsf{d}}}_t \right\|_\rhv^2 + \frac{1}{2\alpha_t} \left\| \bm{x}_{t}-\bm{x}^\star \right\|_\rhv^2,
\end{align*}
which implies that
\begin{align*}
	&\langle\bm{\mathsf{g}}_t ,\bm{x}_{t}-\bm{x}^\star\rangle\\
	\le&\frac{1}{2\alpha_t(1-\beta_{1t})}\left\{\left\| \bm{x}_{t}-\bm{x}^\star\right\|_\rhv^2-\left\|\bm{x}_{t+1}-\bm{x}^\star\right\|_\rhv^2\right\}\\
	&+\frac{\alpha_t}{2(1-\beta_{1t})}\left\{\left\| \hat{\bm{\mathsf{d}}}_t \right\|_\rhv^2+\beta_{1t}\left\| \bar{\bm{\mathsf{d}}}_t \right\|_\rhv^2\right\}\\
	&+\frac{\beta_{1t}}{\alpha_t(1-\beta_{1t})}\left\|\bm{x}_{t}-\bm{x}^\star \right\|_\rhv^2\\
	&+\frac{\gamma_t}{t^a}\left\langle\bm{\mathsf{d}}_{t-1},\bm{x}_{t}-\bm{x}^\star\right\rangle.
\end{align*}
Summing the above inequality from $t=1$ to $T$ ensures that
\begin{align*}
	R(T) =&\sum_{t=1}^T\{f_t(\bm{x}_t-f(\bm{x}^\star)\}\le\sum_{t=1}^T\langle\bm{\mathsf{g}}_t , \bm{x}_{t}-\bm{x}^\star\rangle\\
	\le&\frac{1}{2\alpha_1(1-\beta_1)}\left\| \bm{x}_{1}-\bm{x}^\star\right\|_\rhv^2\\
	&+\frac{1}{2(1-\beta_1)} \sum_{t=2}^T\left\{ \frac{\left\| \bm{x}_{t}-\bm{x}^\star\right\|_\rhv^2}{\alpha_t} - \frac{\left\| \bm{x}_t-\bm{x}^\star\right\|_{\hat{V}_{t-1}^{\frac{1}{2}}}^2}{\alpha_{t-1}} \right\}\\
	&+\frac{\alpha\sqrt{1+\log T}}{(1-\beta_1)^2(1-\mu)\sqrt{1-\beta_2}}\sum_{i=1}^N\sqrt{\sum_{t=1}^{T}\mathsf{d}_{t,i}^2}\\
	&+\sum_{t=1}^T\frac{\beta_{1t}}{2\alpha_t(1-\beta_{1t})}\left\|\bm{x}_t-\bm{x}^\star \right\|_\rhv^2\\
	&+\sum_{t=1}^T \frac{\gamma_t}{t^a}\left\langle\bm{\mathsf{d}}_{t-1},\bm{x}_t-\bm{x}^\star\right\rangle,
\end{align*}
which, together with Lemma \ref{lem:d-bdd}, implies that
\begin{align*}
	R(T) &\leq \frac{1}{2\alpha_1(1-\beta_1)}\sum_{i=1}^N \sqrt{\hat{v}_{1,i}}(x_{1,i}-x_i^\star)^2\\
	&+\frac{1}{2(1-\beta_1)} \sum_{t=2}^T\sum_{i=1}^N(x_{t,i}-x_i^\star)^2\left( \frac{\sqrt{\hat{v}_{t, i}}}{\alpha_t} - \frac{\sqrt{\hat{v}_{t-1, i}}}{\alpha_{t-1}} \right)\\
	&+\frac{\alpha\sqrt{1+\log T}}{(1-\beta_1)^2(1-\mu)\sqrt{1-\beta_2}}\sum_{i=1}^N\sqrt{\sum_{t=1}^{T}\mathsf{d}_{t,i}^2}\\
	&+\frac{1}{2(1-\beta_1)}\sum_{t=1}^T\frac{\beta_{1t}}{\alpha_t}\sum_{i=1}^N \sqrt{\hat{v}_{t,i}}(x_{t,i}-x_i^\star)^2\\
	&+\sum_{t=1}^T \frac{\gamma_t}{t^a}\left\langle\bm{\mathsf{d}}_{t-1},\bm{x}_t-\bm{x}^\star\right\rangle.
\end{align*}
The fact that \(D_\infty\) is the bounded diameter of the feasible set \(\mathcal{F}\) ensures that \(|x_{t,i}-x_i^\star |\le D_\infty\). Therefore, Lemma \ref{lem:d-bdd} guarantees that 
\begin{align*}
	R(T)\le&\frac{D_\infty^2\sqrt{T}}{\alpha(1-\beta_1)}\sum_{i=1}^N \sqrt{\hat{v}_{T,i}}\\
	&+\frac{\alpha\sqrt{1+\log T}}{(1-\beta_1)^2(1-\mu)\sqrt{1-\beta_2}}\sum_{i=1}^N\sqrt{\sum_{t=1}^{T}\mathsf{d}_{t,i}^2}\\
	&+\frac{D_\infty^2}{2(1-\beta_1)}\sum_{t=1}^T\frac{\beta_{1t}}{\alpha_t}\sum_{i=1}^N \sqrt{\hat{v}_{t,i}}
	+D_\infty \bar{G}_\infty\sum_{t=1}^T \frac{|\gamma_t|}{t^a}.
\end{align*}
This completes the proof.
\end{proof}

\section*{Acknowledgments}
We thank Kazuhiro Hishinuma for his input on the numerical evaluation.


\end{document}